\documentclass[11pt,a4paper]{article}

\usepackage[T1]{fontenc}
\usepackage[utf8]{inputenc}
\usepackage{lmodern}
\usepackage{amsmath,amssymb,amsthm,mathtools}
\usepackage{microtype}
\usepackage[margin=27mm]{geometry}
\usepackage{enumitem}
\usepackage{tikz}
\usepackage[hidelinks,pdfencoding=auto]{hyperref}
\usepackage{bookmark}

\hypersetup{
  pdftitle={Absorption probabilities of random walks without symmetry or general position},
  pdfauthor={Artyom Bolotin and Dmitry Zaporozhets}
}

\numberwithin{equation}{section}

\newtheorem{theorem}{Theorem}[section]
\newtheorem{corollary}[theorem]{Corollary}

\theoremstyle{remark}
\newtheorem{remark}[theorem]{Remark}

\newcommand{\R}{\mathbb{R}}

\newcommand{\E}{\mathbb{E}}
\newcommand{\PP}{\mathbb{P}}
\newcommand{\ind}{\mathbf{1}}
\newcommand{\Sym}{\mathfrak{S}}
\newcommand{\cyc}{\mathcal{C}}
\DeclareMathOperator{\conv}{conv}
\newcommand{\stirling}[2]{\genfrac{[}{]}{0pt}{}{#1}{#2}}
\newcommand{\abs}[1]{\lvert #1\rvert}

\newcommand{\eqd}{\stackrel{d}{=}}

\allowdisplaybreaks[1]

\title{Absorption probabilities of random walks\\ without symmetry or general position\\[0.6em]}
\author{Artyom Bolotin \and Dmitry Zaporozhets}
\date{}

\begin{document}
\maketitle

\begin{abstract}
Let $S_1,S_2,\dots$ be a random walk in $\R^d$ with independent identically distributed increments. We show that, for every distribution of the increments and for every $n\ge1$,
\[
\mathbb P\bigl[0\notin\mathrm{conv}(S_1,\dots,S_n)\bigr] =  \sum_{m=1}^n\frac1{m!}\sum_{\substack{k_1,\dots,k_m\ge1\\ k_1+\dots+k_m=n}}\frac{1}{k_1\cdots k_m}\,\mathbb P\bigl[0\notin\mathrm{conv}(S^{(1)}_{k_1},\dots,S^{(m)}_{k_m})\bigr],
\]
where $S^{(1)},S^{(2)},\dots$ are independent copies of the walk.
No symmetry, general position or moment assumptions are needed. The result rests on a deterministic identity: for arbitrary vectors $x_1,\dots,x_n$, the number of orderings for which the convex hull of the partial sums avoids the origin equals the number of permutations for which the convex hull of the cycle sums avoids it. We derive it from a permutation-averaging identity for valuations obtained in our earlier work. Consequences include the identities of Sparre Andersen and Spitzer, the distribution-free absorption probabilities of Kabluchko, Vysotsky and Zaporozhets, their generating function found by Petrov, Randon-Furling and Zaporozhets, and a formula for the probability that the origin is never absorbed.

\end{abstract}



\section{Introduction}\label{sec:intro}

Let $X_1,X_2,\dots$ be independent identically distributed (i.i.d.) random vectors in $\R^d$, and let $S_0=0$ and $S_k=X_1+\dots+X_k$. We say that the origin is \emph{absorbed} by time $n$ if $0\in\conv(S_1,\dots,S_n)$. The boundary of the hull counts; for instance, a visit of the walk to the origin is an absorption. Put
\begin{equation}\label{eq:pn}
p_n=\PP\bigl[0\notin\conv(S_1,\dots,S_n)\bigr],\quad n\ge1,\qquad p_0=1.
\end{equation}

In dimension one the origin is not absorbed if and only if the walk stays strictly positive or strictly negative up to time $n$. The fluctuation theory of Sparre Andersen~\cite{SA49,SA54} and Spitzer~\cite{Spitzer56} expresses these probabilities through the laws of the individual partial sums $S_k$, and for symmetric continuous increments they do not depend on the distribution at all. Kabluchko, Vysotsky and Zaporozhets~\cite{KVZ17a} found the multidimensional form of the last fact: if the increments are centrally symmetric and in general position, then $p_n$ is given by a universal combinatorial expression; see Corollary~\ref{cor:kvz}. The planar case had been treated in~\cite{VZ18}. A proof based on generating functions and the conic Gauss-Bonnet formula was given in~\cite{PRZ23}, and a deformation (wall-crossing) proof has appeared very recently in~\cite{KT26}.

Without symmetry, $p_n$ depends on the distribution of the increments; without general position, boundary effects appear. We show that much of the structure survives when both assumptions are dropped.

For vectors $x_1,\dots,x_n\in\R^d$ and a permutation $\sigma\in\Sym_n$ put
\begin{equation}\label{eq:sums}
s_k(\sigma)=x_{\sigma(1)}+\dots+x_{\sigma(k)},\quad 1\le k\le n,\qquad\quad s_\gamma=\sum_{i\in\gamma}x_i,\quad \gamma\in\cyc(\sigma),
\end{equation}
where $\cyc(\sigma)$ is the set of cycles of $\sigma$, each cycle being regarded as a subset of $\{1,\dots,n\}$. Thus $s_1(\sigma),\dots,s_n(\sigma)$ is the trajectory obtained when the steps are made in the order $\sigma(1),\dots,\sigma(n)$, while the vectors $s_\gamma$ are the sums of the steps along the cycles of $\sigma$.

\begin{theorem}\label{th:det}
For arbitrary $x_1,\dots,x_n\in\R^d$,
\begin{equation}\label{eq:det}
\sum_{\sigma\in\Sym_n}\ind\bigl\{0\notin\conv(s_1(\sigma),\dots,s_n(\sigma))\bigr\}
=\sum_{\sigma\in\Sym_n}\ind\bigl\{0\notin\conv(s_\gamma:\gamma\in\cyc(\sigma))\bigr\}.
\end{equation}
\end{theorem}

In dimension one, a finite set of reals avoids the origin in its convex hull if and only if it lies in $(0,\infty)$ or in $(-\infty,0)$, and~\eqref{eq:det} becomes the passage from orderings to cycles that underlies the identities of Sparre Andersen and Spitzer~\cite{Spitzer56}, \cite[Chapter~XII]{Feller71}. The point of Theorem~\ref{th:det} is that the identity holds in every dimension and for every configuration, degenerate ones included. The boundary counts on both sides: for $n=2$, $x_1=1$ and $x_2=-1$, the path hulls $\conv(1,0)$ and $\conv(-1,0)$ contain the origin on their boundary, while the cycle hulls are $[-1,1]$ and $\{0\}$, so that replacing all hulls in~\eqref{eq:det} by their interiors would give $2=1$.

For i.i.d.\ increments, sums over disjoint cycles are independent and distributed as partial sums of the walk. This leads to the following result, in which $S^{(1)},S^{(2)},\dots$ are independent copies of the walk.

\begin{theorem}\label{th:main}
For every distribution of the increments and every $0\le t<1$,
\begin{equation}\label{eq:gf}
\sum_{n=0}^\infty p_nt^n=1+\sum_{m=1}^\infty\frac1{m!}\sum_{k_1,\dots,k_m\ge1}\frac{t^{k_1+\dots+k_m}}{k_1\cdots k_m}\,\PP\bigl[0\notin\conv(S^{(1)}_{k_1},\dots,S^{(m)}_{k_m})\bigr].
\end{equation}
In particular, for every $n\ge1$,
\begin{equation}\label{eq:main}
p_n=\sum_{m=1}^n\frac1{m!}\sum_{\substack{k_1,\dots,k_m\ge1\\ k_1+\dots+k_m=n}}\frac{1}{k_1\cdots k_m}\,\PP\bigl[0\notin\conv(S^{(1)}_{k_1},\dots,S^{(m)}_{k_m})\bigr].
\end{equation}
Moreover,~\eqref{eq:gf} is equivalent to the following identity. Let $0<t<1$, let $N$ be independent of the walk with $\PP[N=n]=(1-t)t^n$, $n\ge0$, and let $\Pi_t$ be a Poisson point process in $\R^d$ with intensity measure
\begin{equation}\label{eq:nu}
\nu_t=\sum_{k=1}^\infty\frac{t^k}{k}\,\PP[S_k\in\cdot\,].
\end{equation}
Then
\begin{equation}\label{eq:poisson-intro}
\PP\bigl[0\notin\conv(S_1,\dots,S_N)\bigr]=\PP\bigl[0\notin\conv(\Pi_t)\bigr].
\end{equation}
\end{theorem}

Here $\conv(\varnothing)=\varnothing$, and $\Pi_t$ is a finite random multiset, since $\nu_t(\R^d)=\log\frac1{1-t}$. Theorem~\ref{th:main} replaces the dependent positions of the walk by independent points whose laws are those of single partial sums. Symmetry and general position enter only when the probabilities in~\eqref{eq:gf} are evaluated. Under these assumptions Wendel's theorem~\cite{Wendel62} shows that they depend on $m$ alone, and~\eqref{eq:main} becomes
\begin{equation}\label{eq:mixture-intro}
p_n=\E\,w_{K_n,d},\qquad w_{m,d}=2^{1-m}\sum_{j=0}^{d-1}\binom{m-1}{j},
\end{equation}
where $K_n$ is the number of cycles of a uniform random permutation of $\{1,\dots,n\}$. In words, a symmetric walk in general position avoids the origin exactly as often as $K_n$ independent symmetric points do. This gives the formulas of~\cite{KVZ17a} and~\cite{PRZ23} (Section~\ref{sec:symmetric}); the representation~\eqref{eq:mixture-intro} was also obtained in~\cite[Corollary~23]{KT26} by a different route.

Letting $t\uparrow1$ in~\eqref{eq:poisson-intro}, we obtain the probability that the origin is never absorbed.

\begin{theorem}\label{th:never}
For every random walk with i.i.d.\ increments,
\[
\PP\bigl[0\notin\conv\{S_1,S_2,\dots\}\bigr]=\lim_{n\to\infty}p_n=\PP\bigl[0\notin\conv(\Pi_1)\bigr],
\]
where $\Pi_1$ is the superposition of independent Poisson processes with intensities $k^{-1}\PP[S_k\in\cdot\,]$, $k\ge1$, that is, a Poisson process with intensity $\nu_1=\sum_{k\ge1}k^{-1}\PP[S_k\in\cdot\,]$. This superposition defines $\Pi_1$ as a counting measure that need not be locally finite.
\end{theorem}

For $d=1$, Theorem~\ref{th:never} gives the sum of Spitzer's formulas for the probabilities that the walk stays positive forever and that it stays negative forever. The formulas themselves, together with the identities of Sparre Andersen, are recovered in Section~\ref{sec:dim1}.

The new results are Theorems~\ref{th:det}, \ref{th:main} and~\ref{th:never}, which need neither symmetry nor general position; in dimension one they reduce to classical identities. Theorem~\ref{th:det} follows from the permutation-averaging identity for valuations obtained in our earlier paper~\cite{BZ26} (see also~\cite[formula~(32)]{Bol25}), applied to the indicator that a polytope meets a fixed affine subspace. In~\cite{BZ26} the same identity was applied to mixed volumes; this gives the expected mixed volumes of convex hulls of random walks and, in particular, the formulas of~\cite{VZ18} for their expected intrinsic volumes without the assumption that the increments put no mass on affine hyperplanes. The remaining statements are known, and we derive them to show how they fit together.

Section~\ref{sec:deterministic} contains the proof of Theorem~\ref{th:det}. Section~\ref{sec:poisson} proves Theorems~\ref{th:main} and~\ref{th:never}. Section~\ref{sec:dim1} is devoted to dimension one, and Section~\ref{sec:symmetric} derives the formulas for symmetric walks in general position.

\section{The deterministic identity}\label{sec:deterministic}

A real-valued function $\varphi$ on the convex polytopes in $\R^n$ is called a valuation if
\[
\varphi(P\cup P')+\varphi(P\cap P')=\varphi(P)+\varphi(P')
\]
whenever $P$, $P'$ and $P\cup P'$ are convex polytopes. Let $e_1,\dots,e_n$ be the standard basis of $\R^n$, and put $e_B=\sum_{i\in B}e_i$ for $B\subset\{1,\dots,n\}$. With every $\sigma\in\Sym_n$ we associate the simplex and the parallelotope
\[
\Delta_\sigma=\conv\bigl(0,e_{\sigma(1)},e_{\sigma(1)}+e_{\sigma(2)},\dots,e_{\sigma(1)}+\dots+e_{\sigma(n)}\bigr),\qquad Z_\sigma=\sum_{\gamma\in\cyc(\sigma)}[0,e_\gamma].
\]
By the permutation-averaging identity~\cite[Proposition~4.9 and Remark~4.10]{BZ26} (see also~\cite[formula~(32)]{Bol25}), applied to the linearly independent vectors $e_1,\dots,e_n$, every valuation $\varphi$ on the convex polytopes in $\R^n$ satisfies
\begin{equation}\label{eq:BZ}
\sum_{\sigma\in\Sym_n}\varphi(\Delta_\sigma)=\sum_{\sigma\in\Sym_n}\varphi(Z_\sigma).
\end{equation}
Identities of this type go back to the combinatorial lemmas of Spitzer~\cite{Spitzer56} and of Barndorff-Nielsen and Baxter~\cite{BNB63}; see~\cite[Remark~4.10]{BZ26}.

We apply~\eqref{eq:BZ} to valuations of the following form. For a convex set $Q\subset\R^n$, let $\varphi_Q(P)=1$ if $P\cap Q\ne\varnothing$ and $\varphi_Q(P)=0$ otherwise. This is a valuation. Indeed, let $P$, $P'$ and $P\cup P'$ be convex polytopes. If $P$ and $P'$ both meet $Q$, then the segment joining a point of $P\cap Q$ with a point of $P'\cap Q$ lies in $(P\cup P')\cap Q$ and is covered by the closed sets $P$ and $P'$, so it meets $P\cap P'$, and both sides of the valuation identity are equal to~$2$. Otherwise $P\cap P'$ does not meet $Q$, and $\varphi_Q(P\cup P')=\varphi_Q(P)+\varphi_Q(P')$.

\begin{proof}[Proof of Theorem~\ref{th:det}]
Let $A\colon\R^n\to\R^d$ be the linear map with $Ae_i=x_i$, and let
\[
Q=\{t\in\R^n:At=0,\ t_1+\dots+t_n=1\}.
\]
Put $f_k=e_{\sigma(1)}+\dots+e_{\sigma(k)}$, so that $Af_k=s_k(\sigma)$. The points of $\Delta_\sigma$ are the vectors $t=\sum_k\mu_kf_k$ with $\mu_k\ge0$ and $\sum_k\mu_k\le1$, and for them
\[
At=\sum_{k=1}^n\mu_ks_k(\sigma),\qquad t_1+\dots+t_n=\sum_{k=1}^nk\mu_k.
\]
Hence $\Delta_\sigma$ meets $Q$ if and only if $\sum_k\mu_ks_k(\sigma)=0$ for some $\mu_k\ge0$ with $\sum_kk\mu_k=1$; the condition $\sum_k\mu_k\le1$ then holds automatically. Normalizing the coefficients, we see that this is the case if and only if $0\in\conv(s_1(\sigma),\dots,s_n(\sigma))$. In the same way, the points of $Z_\sigma$ are the vectors $t=\sum_\gamma c_\gamma e_\gamma$ with $0\le c_\gamma\le1$, and for them $At=\sum_\gamma c_\gamma s_\gamma$ and $t_1+\dots+t_n=\sum_\gamma\abs\gamma c_\gamma$. The conditions $c_\gamma\le1$ hold automatically when $\sum_\gamma\abs\gamma c_\gamma=1$, and therefore $Z_\sigma$ meets $Q$ if and only if $0\in\conv(s_\gamma:\gamma\in\cyc(\sigma))$. Thus~\eqref{eq:BZ} with $\varphi=\varphi_Q$ says that the convex hulls of the partial sums and of the cycle sums contain the origin for equally many permutations. Subtracting both numbers from $n!$ gives~\eqref{eq:det}.
\end{proof}

\begin{remark}\label{rem:duality}
The proof rests on a duality between the space $\R^d$ of the configuration and the space $\R^n$ of the coefficients: the vectors $x_1,\dots,x_n$ are replaced by the standard basis of $\R^n$, and the origin of $\R^d$ by the affine subspace $Q=\ker A\cap\{t_1+\dots+t_n=1\}$. Since $\Delta_\sigma=\{t:1\ge t_{\sigma(1)}\ge\dots\ge t_{\sigma(n)}\ge0\}$, the origin lies in $\conv(s_1(\sigma),\dots,s_n(\sigma))$ if and only if the linear subspace $\ker A$ meets the cone $\{t:t_{\sigma(1)}\ge\dots\ge t_{\sigma(n)}\ge0\}$ outside the origin. A similar duality was the starting point of~\cite{KVZ17a}. These cones and their images under the sign changes of coordinates are the closed Weyl chambers of type $B_n$, and in~\cite{KVZ17a} the absorption probabilities of symmetric random walks in general position were obtained by counting the Weyl chambers of type $B_n$ met by a generic linear subspace of codimension~$d$.
\end{remark}

\begin{remark}\label{rem:several}
Theorem~\ref{th:det} extends to several walks. Let $x^{(a)}_1,\dots,x^{(a)}_{n_a}\in\R^d$, $a=1,\dots,r$, and let $\varphi$ be a valuation on the convex polytopes in $\R^{n_1}\times\dots\times\R^{n_r}$. If the polytopes in all factors but one are fixed, then $\varphi$ of their product with a polytope $P$ in the remaining factor is a valuation in $P$. Applying~\eqref{eq:BZ} in one factor at a time, we see that the sums of $\varphi(\Delta_{\sigma_1}\times\dots\times\Delta_{\sigma_r})$ and of $\varphi(Z_{\sigma_1}\times\dots\times Z_{\sigma_r})$ over $\sigma_1\in\Sym_{n_1},\dots,\sigma_r\in\Sym_{n_r}$ coincide. As in the proof of Theorem~\ref{th:det}, it follows that the number of tuples $(\sigma_1,\dots,\sigma_r)$ for which the joint convex hull of all partial sums $s^{(a)}_k(\sigma_a)$ avoids the origin equals the number of tuples for which the joint convex hull of all cycle sums $s^{(a)}_\gamma$, $\gamma\in\cyc(\sigma_a)$, $1\le a\le r$, avoids it. Averaging over independent uniform permutations as in Section~\ref{sec:poisson}, one obtains absorption probabilities for several independent random walks $S^{(1)},\dots,S^{(r)}$ with i.i.d.\ increments. If the increments are symmetric and put no mass on affine hyperplanes, Wendel's theorem (Section~\ref{sec:symmetric}) shows that the joint convex hull of all positions $S^{(a)}_k$, $1\le k\le n_a$, $1\le a\le r$, avoids the origin with probability $\E\,w_{K^{(1)}+\dots+K^{(r)},d}$, where $K^{(1)},\dots,K^{(r)}$ are independent and $K^{(a)}$ is the number of cycles of a uniform random permutation of $\{1,\dots,n_a\}$. This is the formula of~\cite{KVZ17a} for several walks.
\end{remark}

\begin{remark}\label{rem:reflection}
The simplex $\Delta_\sigma$ is the intersection of the cube $[0,1]^n$ with the closed chamber $\{t:t_{\sigma(1)}\ge\dots\ge t_{\sigma(n)}\}$ of the braid arrangement formed by the hyperplanes $\{t_i=t_j\}$, and $Z_\sigma$ is the intersection of the cube with the space of vectors fixed by the coordinate permutation $\sigma$. The braid arrangement is the reflection arrangement of the symmetric group $\Sym_n$, and the arrangement in $\R^{n_1}\times\dots\times\R^{n_r}$ formed by the hyperplanes of the factors is the reflection arrangement of the product group $\Sym_{n_1}\times\dots\times\Sym_{n_r}$; this is why products appear in Remark~\ref{rem:several}. More generally, let $W$ be a finite real reflection group acting on $\R^n$, let $C$ be a closed fundamental chamber, and let $P\subset\R^n$ be a convex polytope. For every valuation $\varphi$ with $\varphi(\varnothing)=0$,
\[
\sum_{w\in W}\varphi(P\cap wC)=\sum_{w\in W}\varphi(P\cap\operatorname{Fix}(w)),
\qquad \operatorname{Fix}(w)=\{x\in\R^n:wx=x\}.
\]
Indeed, at each point the number of closed chambers containing it equals the number of group elements fixing it. Intersecting with $P$ and applying the valuation property on a common polytopal subdivision gives the displayed identity. Analogous absorption results can be obtained in this setting. We do not pursue this here and will return to it elsewhere.
\end{remark}

\section{Random walks: proofs}\label{sec:poisson}

\begin{proof}[Proof of Theorem~\ref{th:main}]
First let $X_1,\dots,X_n$ be exchangeable random vectors (symmetrically dependent in the terminology of~\cite{SA53}), let $\sigma$ be a uniform random permutation of $\{1,\dots,n\}$ independent of them, and let $S_\gamma=\sum_{i\in\gamma}X_i$ for $\gamma\in\cyc(\sigma)$. By exchangeability, every term on the left-hand side of~\eqref{eq:det}, evaluated at $x_i=X_i$, has the same distribution as the term with $\sigma=\mathrm{id}$. Taking expectations and dividing by $n!$, we obtain
\begin{equation}\label{eq:exch-abs}
\PP\bigl[0\notin\conv(S_1,\dots,S_n)\bigr]=\PP\bigl[0\notin\conv(S_\gamma:\gamma\in\cyc(\sigma))\bigr].
\end{equation}
The events involved are measurable, since the set of configurations $(z_1,\dots,z_m)$ with $0\in\conv(z_1,\dots,z_m)$ is closed.

Now let the increments be i.i.d., let $0<t<1$, let $N$ be independent of the walk with $\PP[N=n]=(1-t)t^n$, $n\ge0$, and, given $N$, let $\sigma$ be uniform on $\Sym_N$. By Cauchy's formula (see~\cite{ABT03}), given $N=n$, the permutation $\sigma$ has $i_k$ cycles of length $k$ for all $k$ with probability $\prod_k(i_k!\,k^{i_k})^{-1}$ whenever $\sum_kki_k=n$. Multiplying by $(1-t)t^n=\prod_ke^{-t^k/k}\prod_k(t^k)^{i_k}$, we see, following Shepp and Lloyd~\cite{SL66}, that the numbers of cycles of lengths $1,2,3,\dots$ are independent Poisson random variables with means $t,t^2/2,t^3/3,\dots$ Given the cycles of $\sigma$, the cycle sums $S_\gamma$ are independent, because the cycles are disjoint, and $S_\gamma$ has the law of $S_{\abs\gamma}$. By the marking and superposition theorems for Poisson processes~\cite[Theorems~3.3 and~5.6]{LP17}, the multiset $\{S_\gamma:\gamma\in\cyc(\sigma)\}$ is therefore a Poisson process with intensity $\sum_k(t^k/k)\PP[S_k\in\cdot\,]=\nu_t$. Applying~\eqref{eq:exch-abs} conditionally on $N$, we obtain~\eqref{eq:poisson-intro}.

The left-hand side of~\eqref{eq:poisson-intro} equals $(1-t)\sum_{n\ge0}p_nt^n$. On the right-hand side, $\Pi_t$ consists of a Poisson number of points with mean $\log\frac1{1-t}$, which, given their number, are i.i.d.\ with law $\nu_t/\nu_t(\R^d)$. Expanding $\nu_t$ as a sum over $k$ and dividing by $1-t$, we see that~\eqref{eq:poisson-intro} is equivalent to~\eqref{eq:gf}. All coefficients in~\eqref{eq:gf} are nonnegative, and comparing the coefficients of $t^n$ gives~\eqref{eq:main}.
\end{proof}

\begin{remark}\label{rem:blocks}
The first step of the proof uses only exchangeability, so~\eqref{eq:exch-abs} reduces $p_n$ to cycle sums also for exchangeable increments that are not independent. By exchangeability, the cycle sums of $\sigma$ have, given the cycle lengths $k_1,\dots,k_m$ listed in any order, the joint law of the consecutive block sums $S_{b_1},S_{b_2}-S_{b_1},\dots,S_{b_m}-S_{b_{m-1}}$ with $b_j=k_1+\dots+k_j$. Hence~\eqref{eq:main} holds with $\PP[0\notin\conv(S^{(1)}_{k_1},\dots,S^{(m)}_{k_m})]$ replaced by the probability that the convex hull of these block sums avoids the origin.
\end{remark}

\begin{proof}[Proof of Theorem~\ref{th:never}]
The hulls $\conv(S_1,\dots,S_n)$ increase with union $\conv\{S_1,S_2,\dots\}$, so $p_n$ decreases to the left-hand side, and $(1-t)\sum_np_nt^n\to\lim_np_n$ as $t\uparrow1$ by Abel's theorem. Realize all $\Pi_t$ on one probability space: take a Poisson process $\Pi$ on $\{1,2,\dots\}\times\R^d\times[0,1]$ with intensity $\sum_kk^{-1}\delta_k\otimes\PP[S_k\in\cdot\,]\otimes du$, and let $\Pi_t$ consist of the points $z$ of those triples $(k,z,u)\in\Pi$ with $u<t^k$. By the thinning theorem~\cite[Chapter~5]{LP17}, $\Pi_t$ has intensity $\nu_t$, and $\Pi_t$ increases to $\Pi_1$ as $t\uparrow1$. The origin lies in $\conv(\Pi_1)$ if and only if it lies in the convex hull of a finite subset of $\Pi_1$, that is, in $\conv(\Pi_t)$ for some $t<1$. By monotone convergence, $\PP[0\notin\conv(\Pi_t)]\downarrow\PP[0\notin\conv(\Pi_1)]$, and~\eqref{eq:poisson-intro} completes the proof.
\end{proof}

\section{Dimension one}\label{sec:dim1}

Let $d=1$ and let the increments be i.i.d. A finite set of reals avoids the origin in its convex hull if and only if it lies in $(0,\infty)$ or in $(-\infty,0)$. Hence
\[
\PP\bigl[0\notin\conv(S^{(1)}_{k_1},\dots,S^{(m)}_{k_m})\bigr]=\prod_{j=1}^m\PP[S_{k_j}>0]+\prod_{j=1}^m\PP[S_{k_j}<0],
\]
and~\eqref{eq:gf} becomes the two-sided form of Sparre Andersen's identity,
\[
\sum_{n=0}^\infty p_nt^n=\exp\Bigl\{\sum_{k=1}^\infty\frac{t^k}k\PP[S_k>0]\Bigr\}+\exp\Bigl\{\sum_{k=1}^\infty\frac{t^k}k\PP[S_k<0]\Bigr\}-1.
\]
For the one-sided form, replace $Q$ in the proof of Theorem~\ref{th:det} by the convex set $\{t\in\R^n:At\le0,\ t_1+\dots+t_n=1\}$. The same computation shows that $\Delta_\sigma$ meets this set if and only if $s_k(\sigma)\le0$ for some $k$, and $Z_\sigma$ meets it if and only if $s_\gamma\le0$ for some $\gamma\in\cyc(\sigma)$. Hence
\[
\sum_{\sigma\in\Sym_n}\ind\bigl\{s_1(\sigma)>0,\dots,s_n(\sigma)>0\bigr\}=\sum_{\sigma\in\Sym_n}\ind\bigl\{s_\gamma>0\text{ for all }\gamma\in\cyc(\sigma)\bigr\}.
\]
Arguing as in the proof of Theorem~\ref{th:main}, we find that $\PP[S_1>0,\dots,S_N>0]$ equals the probability that $\Pi_t$ has no points in $(-\infty,0]$, that is, $\exp\{-\sum_k(t^k/k)\PP[S_k\le0]\}$. Dividing by $1-t=\exp\{-\sum_kt^k/k\}$, we obtain the identity of Sparre Andersen~\cite{SA54}
\begin{equation}\label{eq:SAstrict}
\sum_{n=0}^\infty t^n\PP[S_1>0,\dots,S_n>0]=\exp\Bigl\{\sum_{k=1}^\infty\frac{t^k}k\PP[S_k>0]\Bigr\}.
\end{equation}
Applying~\eqref{eq:SAstrict} to the increments $X_i+\varepsilon$ and letting $\varepsilon\downarrow0$, we obtain the same identity with $\ge$ in place of $>$, since the events $\{S_k>-k\varepsilon$ for $k\le n\}$ decrease to $\{S_k\ge0$ for $k\le n\}$. Multiplying~\eqref{eq:SAstrict} by $1-t$ and letting $t\uparrow1$, by Abel's theorem and monotone convergence we obtain Spitzer's formula~\cite{Spitzer56}, \cite[Chapter~XII]{Feller71}
\[
\PP[S_n>0\text{ for all }n]=\exp\Bigl\{-\sum_{k=1}^\infty\frac{\PP[S_k\le0]}k\Bigr\},
\]
and, applied to the increments $-X_i$, the analogous formula for $\PP[S_n<0\text{ for all }n]$. The sum of the two formulas is the case $d=1$ of Theorem~\ref{th:never}.

\section{Symmetric walks in general position}\label{sec:symmetric}

We now derive the distribution-free formulas of~\cite{KVZ17a} and~\cite{PRZ23} from Theorem~\ref{th:main}. The only additional ingredient is Wendel's theorem~\cite{Wendel62}: if $Y_1,\dots,Y_m$ are random vectors in $\R^d$ whose joint law is invariant under all sign changes $Y_i\mapsto\pm Y_i$, and any $\min(m,d)$ of them are linearly independent almost surely, then
\begin{equation}\label{eq:wendel}
\PP\bigl[0\notin\conv(Y_1,\dots,Y_m)\bigr]=w_{m,d}=2^{1-m}\sum_{j=0}^{d-1}\binom{m-1}j.
\end{equation}
Wendel's original formulation already allows joint invariance under individual sign changes; see~\cite[pp.~109--110]{Wendel62} and also~\cite[Theorem~7]{KT26}. On the walk we impose the general position condition
\begin{description}[font=\normalfont,leftmargin=3.2em]
\item[(GP)] for all $1\le r\le\min(d,n)$ and $1\le i_1<\dots<i_r\le n$, the vectors $S_{i_1},\dots,S_{i_r}$ are almost surely linearly independent.
\end{description}
It holds, for instance, if $X_1$ puts no mass on affine hyperplanes: consecutive block sums of the walk are then independent and put no mass on affine hyperplanes, so each of them avoids the span of the previous ones almost surely.

Recall that $K_n$ is the number of cycles of a uniform random permutation of $\{1,\dots,n\}$. Then $\PP[K_n=m]=\stirling nm/n!$, where the unsigned Stirling numbers of the first kind are defined by $\sum_m\stirling nmu^m=u(u+1)\cdots(u+n-1)$, and, with $h(t)=\log\frac1{1-t}$,
\begin{equation}\label{eq:Kn}
\sum_{n=1}^\infty\PP[K_n=m]\,t^n=\frac{h(t)^m}{m!}=\frac1{m!}\sum_{k_1,\dots,k_m\ge1}\frac{t^{k_1+\dots+k_m}}{k_1\cdots k_m},\qquad m\ge1.
\end{equation}
Finally, let $B_n(u)=\prod_{r=1}^n(u+2r-1)=\sum_{j=0}^nB(n,j)\,u^j$, and $B(n,j)=0$ for $j\notin\{0,\dots,n\}$.

\begin{corollary}[\cite{KVZ17a}]\label{cor:kvz}
Let the increments be i.i.d.\ and symmetric, $X_1\eqd-X_1$, and let \textup{(GP)} hold. Then
\[
p_n=\E\,w_{K_n,d}=\frac2{2^nn!}\sum_{\ell\ge0}B(n,d-1-2\ell).
\]
\end{corollary}

\begin{proof}
Let $k_1,\dots,k_m\ge1$ and $k_1+\dots+k_m=n$. The vectors $S^{(1)}_{k_1},\dots,S^{(m)}_{k_m}$ are independent and symmetric, so their joint law is invariant under sign changes. Any $r\le\min(m,d)$ of them have the joint law of consecutive block sums $S_{b_1},S_{b_2}-S_{b_1},\dots,S_{b_r}-S_{b_{r-1}}$ of the walk with $b_r\le n$. These are obtained from $S_{b_1},\dots,S_{b_r}$ by an invertible linear transformation and are therefore linearly independent almost surely by (GP). By Wendel's theorem, every probability in~\eqref{eq:main} equals $w_{m,d}$, and by~\eqref{eq:Kn} the remaining sum over $k_1,\dots,k_m$ equals $\PP[K_n=m]$. This proves that $p_n=\E\,w_{K_n,d}$.

For the second expression, note that $\sum_{j<d}\binom{m-1}j=\sum_{j\le d-1,\,j\equiv d-1}\binom mj$ by Pascal's rule, where $\equiv$ refers to congruence modulo $2$, and that
\[
\sum_m\stirling nm2^{n-m}(1+u)^m=2^n\prod_{r=0}^{n-1}\Bigl(\frac{1+u}2+r\Bigr)=B_n(u).
\]
Hence $2^nn!\,\E\,w_{K_n,d}=2\sum_{j\le d-1,\,j\equiv d-1}B(n,j)$.
\end{proof}

The same argument applies to symmetrically exchangeable increments, that is, when $(\varepsilon_1X_{\tau(1)},\dots,\varepsilon_nX_{\tau(n)})$ has the law of $(X_1,\dots,X_n)$ for all $\tau\in\Sym_n$ and all signs $\varepsilon_i$; this is the setting of~\cite{KVZ17a}. By Remark~\ref{rem:blocks}, the independent sums in~\eqref{eq:main} are then replaced by block sums, and changing the signs of all increments in one block changes the sign of this block sum alone, so Wendel's theorem still applies. The mixture $\E\,w_{K_n,d}$ also appears in~\cite[Corollary~23]{KT26}, where marking the cycles connects it with the representation $\PP[0\in\conv(S_1,\dots,S_n)]=\PP[U_2+\dots+U_n\ge d]$ with independent Bernoulli variables $U_m$ of parameter $1/(2m)$, see~\cite[Proposition~21]{KT26}.

\begin{corollary}[\cite{PRZ23}]\label{cor:prz}
If the increments are i.i.d.\ and symmetric and \textup{(GP)} holds for every $n$, for instance if $X_1$ is symmetric and puts no mass on affine hyperplanes, then
\begin{equation}\label{eq:prz}
\sum_{n=0}^\infty p_nt^n=(-1)^d+\frac2{\sqrt{1-t}}\sum_{\substack{0\le j\le d-1\\ j\equiv d-1\ (\mathrm{mod}\ 2)}}\frac{(h(t)/2)^j}{j!},\qquad 0\le t<1.
\end{equation}
\end{corollary}

\begin{proof}
As in the proof of Corollary~\ref{cor:kvz}, the probabilities in~\eqref{eq:gf} are equal to $w_{m,d}$, so~\eqref{eq:gf} becomes $\sum_np_nt^n=1+\sum_{m\ge1}w_{m,d}\,h(t)^m/m!$. Write $w_{m,d}=2^{1-m}\sum_{j\le d-1,\,j\equiv d-1}\binom mj$ and use
\[
\sum_{m=0}^\infty\binom mj\frac{(h/2)^m}{m!}=\frac{(h/2)^j}{j!}\,e^{h/2},\qquad e^{h(t)/2}=\frac1{\sqrt{1-t}}.
\]
The term $m=0$ has to be removed; it contributes $2$ if $j=0$, that is, if $d$ is odd. Since $1-2\cdot\ind\{d\text{ odd}\}=(-1)^d$, the claim follows.
\end{proof}

In~\cite{PRZ23} the right-hand side of~\eqref{eq:prz} is written with $h(t)^j/(2j)!!$, which is the same as $(h(t)/2)^j/j!$, and without the additive correction $(-1)^d$. This does not affect the coefficients of $t^n$ with $n\ge1$; the additive correction $(-1)^d$ makes the coefficient of $t^0$ equal to $p_0=1$. For example, for $d=2$ the generating function is $1+h(t)/\sqrt{1-t}$.

\section*{Declaration of generative AI and AI-assisted technologies}

The main results of this paper were obtained in the course of a long and continuous discussion with the large language model Astra (OpenAI). The permutation-averaging identity for valuations on which the proof of Theorem~\ref{th:det} rests was obtained earlier, without the use of AI tools, in the bachelor's thesis of the first author~\cite{Bol25}; see also~\cite{BZ26}. The text was prepared and edited with the assistance of Claude Opus 5.5 (Anthropic). The authors reviewed and edited the content as needed and take full responsibility for the content of the publication.

\bigskip

\noindent
Artyom Bolotin\\
St. Petersburg Department of the Steklov Mathematical Institute
of the Russian Academy of Sciences,\\
27 Fontanka, St. Petersburg 191023, Russia\\
St. Petersburg State University,
7/9 Universitetskaya Emb., St. Petersburg 199034, Russia\\
\texttt{bolotin2003@yandex.ru}

\bigskip

\noindent
Dmitry Zaporozhets\\
St. Petersburg Department of the Steklov Mathematical Institute
of the Russian Academy of Sciences
27 Fontanka, St. Petersburg 191023, Russia\\
St. Petersburg State University,
7/9 Universitetskaya Emb., St. Petersburg 199034, Russia\\
\texttt{zap1979@gmail.com}

\end{document}